\documentclass[12pt]{article}
\usepackage[utf8]{inputenc}
\usepackage{enumitem}
\usepackage{amsmath}
\usepackage{amssymb}
\usepackage{booktabs}
\usepackage{setspace}
\usepackage{graphicx}
\usepackage{amsthm}
\usepackage{float}
\usepackage{subcaption}
\usepackage{geometry}
\usepackage{float}
\usepackage[affil-it]{authblk}
\usepackage[authoryear,round]{natbib}
\makeatletter 
\renewcommand\NAT@biblabel[1]{[#1]}

\usepackage[affil-it]{authblk}

\usepackage{authblk}

\newtheorem{thm}{Theorem}
\newtheorem{lem}{Lemma}
\newtheorem{cor}{Corollary}
\newtheorem{prop}{Proposition}
\newtheorem{as}{Assumption}
\newtheorem{rem}{Remark}
\newcommand{\htau}{\hat{\tau}}
\newcommand{\ttau}{\tilde{\tau}}
\newcommand{\fn}{\frac{1}{n}}
\newcommand{\htheta}{\hat{\Theta}}

\newcommand{\shsigma}{\hat{\sigma}}
\newcommand{\lhs}{\hat{S}}
\newcommand{\shs}{\hat{s}}

\newcommand{\shgamma}{\hat{\gamma}}
\newcommand{\homega}{\hat{\Omega}}
\newcommand{\hsigma}{\hat{\Sigma}}
\newcommand{\hbeta}{\hat{\beta}}
\newcommand{\sbhat}{\hat{b}}
\newcommand{\shbeta}{\hat{\beta}}
\newcommand{\argmin}{\mathop{\rm argmin}\limits}

\newcommand*{\transpose}{%
  {\mathpalette\@transpose{}}%
  }
  
  \expandafter\let\expandafter\oldproof\csname\string\proof\endcsname
\let\oldendproof\endproof
\renewenvironment{proof}[1][\proofname]{%
  \oldproof[\bfseries \scshape #1]%
}{\oldendproof}

\usepackage{color}

\title{Small Tuning Parameter Selection for the Debiased Lasso}
\author{Akira Shinkyu\thanks{Email: akira.shinkyu.5589@gmail.com}\ \ and   Naoya Sueishi\thanks{Email: sueishi@econ.kobe-u.ac.jp}}
\providecommand{\keywords}[1]{\textbf{Keywords}: #1}
\affil{Kobe University\thanks{Graduate School of Economics, 2-1 Rokkodai-cho, Nada-ku, Kobe, 657-8501, JAPAN.}}
\date{\today}
 \allowdisplaybreaks
\begin{document}
\maketitle
\begin{abstract}
In this study, we investigate the bias and variance properties of the debiased Lasso in linear regression when the tuning parameter of the node-wise Lasso is selected to be smaller than in previous studies. 
We consider the case where the number of covariates $p$ is bounded by a constant multiple of the sample size $n$.
First, we show that the bias of the debiased Lasso can be reduced without diverging the asymptotic variance by setting the order of the tuning parameter to $1/\sqrt{n}$.
This implies that the debiased Lasso has asymptotic normality provided that the number of nonzero coefficients $s_0$ satisfies $s_0=o(\sqrt{n/\log p})$, whereas previous studies require $s_0 =o(\sqrt{n}/\log p)$ if no sparsity assumption is imposed on the precision matrix.
Second, we propose a data-driven tuning parameter selection procedure for the node-wise Lasso that is consistent with our theoretical results. 
Simulation studies show that our procedure yields confidence intervals with good coverage properties in various settings. 
We also present a real economic data example to demonstrate the efficacy of our selection procedure.
\end{abstract}

\section{Introduction}
In this study, we make inferences about individual coefficients of the linear regression model:  
\begin{equation}
    Y=X\beta_0+\epsilon, \label{100}
\end{equation}
where $Y \in \mathbb{R}^n$ denotes the response, $X=(X_1, \dots, X_p) \in \mathbb{R}^{n \times p}$ denotes the random design matrix, $\beta_0\in \mathbb{R}^p$ denotes the unknown regression coefficient, and $\epsilon\in \mathbb{R}^n \sim N(0, \sigma_\epsilon^2 I)$ denotes the unobservable noise that is independent of $X$.
We consider $p >n$ and assume that $X$ has zero mean and a positive definite covariance matrix $\Sigma=E[X^TX/n]$.
We focus on statistical inferences for $\beta_{01}$, the first component of $\beta_0$, although the inferences are equally valid for other components. \par

Much research has been conducted on high-dimensional linear regression models.
The Lasso (\citealt{t1996}) is one of the most popular methods for estimating their coefficients. The popularity is due to variable selection properties, the oracle inequalities (\citealt{brt2009}, \citealt{bv2011}), and useful computational algorithms (\citealt{ehjt2004}, \citealt{fht2010}). 
\cite{bv2011} gave excellent overviews of the Lasso.
However, the asymptotic distribution of the Lasso is generally intractable because the Lasso is biased and the distribution is not continuous (\citealt{kf2000}).
Therefore, it is difficult to perform statistical inference based on the Lasso itself, and there are currently three main approaches of performing statistical inference in high-dimensional linear regression models. \par

The first approach of inference is the so-called selective inference. 
See \cite{bbbzz2013}, \cite{lsst2016}, \cite{ttlt2016}, and \cite{trtw2018}, among others. 
Statistical inferences after model selection are distorted if the selected model is treated as if it is a true one.
Rather than making inferences on coefficients of the true regression model, the selective inference makes inferences on coefficients in a model selected via data-driven procedures.
The validity of the inferences does not depend on the correctness of the selected model.
\par

The second approach is the orthogonalization approach, which is developed in econometrics. 
Some contributions include the post-double selection method of \cite{bch2014}, projection onto the double selection method of \cite{whx2020}, and double machine learning method of \cite{ccddhnr2018}.
Such methods perform inference on low-dimensional parameters when high-dimensional nuisance parameters appear in regression models. 
The estimators of the parameter of interest are derived through the orthogonal score function, which is insensitive to the estimation of nuisance parameters. \par

The third approach is to debias or desparsify the Lasso, which we consider in this study. 
The debiased Lasso was proposed by \cite{zz2014} and further developed by \cite{vbrd2014} and \cite{jm2014}. 
Unlike selective inference, this approach is intended to make inferences on coefficients in the true linear regression model.
If the true data generating process is sufficiently sparse, the estimator is asymptotically normally distributed, so that inference is possible even when $p>n$.
 \par

The performance of the debiased Lasso depends on the degree of sparsity and the choice of tuning parameters.
To obtain the debiased Lasso estimator of $\beta_{01}$, we need an estimate of $\Theta_1$, the first column of $\Sigma^{-1}$, and the node-wise Lasso (\citealt{mb2006}) is commonly used to obtain it. 
Therefore, the debiased Lasso estimator depends on the tuning parameters of the Lasso $\lambda_0$ and node-wise Lasso $\lambda_1$.  
Let $s_0$ and $s_1$ be the numbers of nonzero components in $\beta_0$ and $\Theta_1$, respectively.
If both $\lambda_0$ and $\lambda_1$ are of order $\sqrt{\log p/n}$, the asymptotic normality of the debiased Lasso is shown for $s_0=o(\sqrt{n}/\log p)$ and $s_1=o(n/\log p)$ (\citealt{vbrd2014}, \citealt{zz2014}) or for $s_0=o(n/\|\Theta_1\|_1^2(\log p)^2)$ and $\min \{s_0,s_1\}=o(\sqrt{n}/\log p)$ (\citealt{jm2018}). 
Recently, \cite{bz2022} showed asymptotic normality of the debiased Lasso  when $\max\{s_0,s_1\}=o(n/\log p)$ and $\min\{s_0,s_1\}=o(\sqrt{n}/\log p)$ using degrees of freedom adjustments.\par

Although the above studies significantly contributed to the development of the debiased Lasso, there are two main issues to be addressed. 
First, the bias of the debiased Lasso may not be asymptotically negligible if $\Theta_1$ and $\beta_0$ are not sufficiently sparse. 
If no sparsity condition is imposed on $\Theta_1$, $s_0$ must satisfy $s_0 = o(\sqrt{n}/\log p)$ for zero-mean asymptotic normality, which is rather restrictive given that only $s_0=o(n/\log p)$ is need for consistent estimation.
Second, existing tuning parameter selection procedures of $\lambda_1$, such as the cross-validation (\citealt{clc2021}) and the methods of \cite{zz2014} and \cite{dbmm2015}, often produce a rather large bias in the debiased Lasso, yielding a poor coverage of confidence intervals and large size distortion of $t$ tests.
 \par

We address these two issues by selecting a smaller tuning parameter of the node-wise Lasso than in prior studies.
First, we show that by setting the order of $\lambda_1$ to $1/\sqrt{n}$, the bias of the debiased Lasso can be reduced without diverging the asymptotic variance, so that asymptotic normality holds provided $s_0=o(\sqrt{n/\log p})$. 
The novelty of our method is that we intentionally overfit the node-wise Lasso to reduce the bias of the debiased Lasso. 
Further, our analysis does not impose any sparsity assumption on $\Theta_1$.
To the best of our knowledge, there is no asymptotic normality result of the debiased Lasso when the order of $\lambda_1$ is $1/\sqrt{n}$.
Moreover, with the exception of \cite{v2019}, no prior study, including \cite{bch2014} and \cite{ccddhnr2018}, has established asymptotic normality without imposing any sparsity condition on $\Theta_1$.
A sufficient condition for our result is that each row vector of $X$ is independent Gaussian and $p \leq C n$ for some $C< \infty$.
Although the condition that $p$ and $n$ be of the same order excludes some high-dimensional data, such as genomic data, various data fit the condition in many fields, such as economics.

Second, we propose a data-driven procedure to select $\lambda_1$, which we call the small tuning parameter selector (STPS). 
We modify the selection method of \cite{zz2014} so that the selector chooses a tuning parameter with a reasonably small value.
Figure 1 compares the distributions of the debiased Lasso whose $\lambda_1$ is selected by cross-validation and STPS, respectively, indicating that STPS successfully reduces the bias.
Simulation studies show that the debiased Lasso tuned by our procedure outperforms other inference methods, such as the double machine learning, post-double selection, and debiased Lasso with other tuning parameter selection methods. 
This is particularly true when $\beta_0$ and $\Theta_1$ are dense. 
We also show a real economic data example to see the performance of our method.
Because the true sparsity of $\beta_0$ and $\Theta_1$ is unknown, it is essential to propose a selection method that is robust to the violation of the sparsity condition. \par
\begin{figure}[ht]
\centering
\begin{subfigure}[t]{.49\linewidth} 
\centering
\caption{10-fold cross-validation.} 
 \includegraphics[scale=0.65]{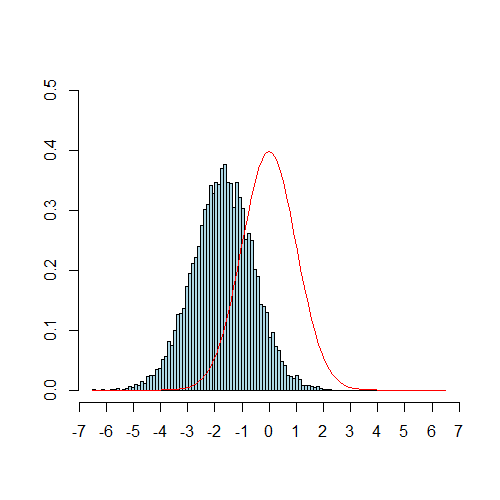}
 \end{subfigure}
 \begin{subfigure}[t]{.49\linewidth}
 \centering
 \caption{STPS.}
 \includegraphics[scale=0.65]{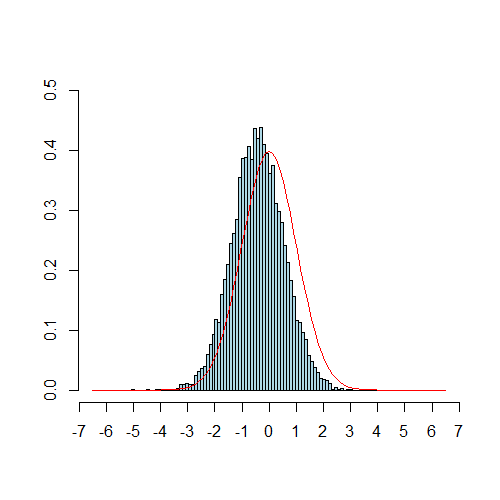}
 \end{subfigure}
  \caption{Distributions of the studentized debiased Lasso (histogram) and the standard normal distribution (solid line). We employ the simulation design II of Table II in \cite{k2020} with $n=200$ and $p=400$ and repeat experiments 10,000 times. The tuning parameter $\lambda_0$ of the Lasso to be debiased is selected by the 10-fold cross-validation. The error variance is estimated by the sample mean of squared residuals from the Lasso tuned by the one standard error rule. For panel (a), the 10-fold cross-validation selects $0.13$ on average, and the coverage of 95\% confidence interval is 59.7\% with the average length of 0.208. For panel (b), STPS selects $0.001$ on average, and the coverage of 95\% confidence interval is 94.3\% with the average length of 0.459.}
\end{figure}

We emphasize that the aim of this study is not to obtain an asymptotically efficient estimator (see, e.g.,  \citealt{vbrd2014}, \citealt{jv2018}, \citealt{v2019}, and \citealt{bz2022} about the asymptotic efficiency of the debiased Lasso).
The aim is to construct confidence intervals with good coverage properties or test statistics with the correct size.
There is a bias--variance tradeoff with respect to $\lambda_1$.
Our method reduces the bias at the cost of slightly increasing the variance.
Setting the order of $\lambda_1$ to $1/\sqrt{n}$ maintains the $\sqrt{n}$-consistency of the debiased Lasso, although the asymptotic variance may exceed the efficiency bound.
\par

The remainder of this article is organized as follows. In Section 2, we formally state an asymptotic normality result of the debiased Lasso when the order of $\lambda_1$ is $1/\sqrt{n}$. 
In Section 3, we propose a data-driven method to select $\lambda_1$. 
In Section 4, we show results of extensive simulation studies for the debiased Lasso tuned by our procedure and other statistical inference methods. We present a real economic data example based on the considered methods in Section 5. Finally, concluding remarks are presented in Section 6. 
All proofs are given in Appendix.

\section{Theoretical Results}
We define $|S|$ as the cardinality of a set $S$. Let $a^T$ be the transpose of a vector $a$. 
Let $\|a\|$, $\|a\|_1$, and $\|a\|_{\infty}$ be the Euclidean norm, $\ell_1$ norm, and maximum element in the absolute value of a vector $a$, respectively. For any sequences of real numbers $\{a_n\}$ and $\{b_n\}$, $a_n \asymp b_n$ means $C_1\leq a_n/b_n \leq C_2$ for all $n$ and some $C_1, C_2 >0$.
Let $a_n \ll b_n$ and $b_n \gg a_n$ denote $\lim_{n\rightarrow \infty } |a_n/b_n|=0$. 
Let $a_n \lesssim b_n$ and $b_n \gtrsim a_n$ denote $a_n\leq Cb_n$ for all $n$ and some $C>0$. 
Let us denote $\lambda_{\min} (A)$ and $\lambda_{\max}(A)$ as the minimum and maximum eigenvalues for a matrix $A$, respectively.

\subsection{Bias of the Debiased Lasso}
First, we introduce the debiased Lasso. 
The Lasso for $\beta_0$ is defined as follows:
\begin{equation}
\hbeta:= \argmin_{\beta \in \mathbb{R}^p}\fn \|Y-X\beta\|^2+2\lambda_0\|\beta\|_1, \label{111}
\end{equation}
where $\lambda_0>0$ denotes the tuning parameter. 
The debiased Lasso $\sbhat_1$ for $\beta_{01}$ is given by
\begin{equation}
    \sbhat_1=\hbeta_1+\frac{1}{n}\htheta_1{\! \! }^TX^T(Y-X\hbeta), \label{115}
\end{equation}
where $\hbeta_1$ denotes the first component of $\hbeta$, and $\htheta_1$ denotes an estimate of $\Theta_1$.

Although it is not a requisite, $\htheta_1$ is commonly constructed by the node-wise Lasso.
Let $X_{-1}$ be a submatrix of $X$ without $X_1$. 
We define
\begin{equation*}
    \shgamma_1:=\argmin_{\gamma \in \mathbb{R}^{p-1}}\fn\|X_1-X_{-1}\gamma\|^2+2\lambda_1\|\gamma\|_1,
\end{equation*}
where $\lambda_1>0$ is another tuning parameter. 
Let $\tau_1^2$ denote the error variance in regressing $X_1$ on $X_{-1}$.
Moreover, let us define $\ttau_1^2$ and $\htau_1^2$ as
\begin{equation*}
    \ttau_1^2:=\fn \|X_1-X_{-1}\shgamma_1\|^2,\ \ \htau_1^2:=\ttau_1^2+\lambda_1\|\shgamma_1\|_1.
\end{equation*}
The Karush-Kuhn--Tucker (KKT) conditions for the node-wise Lasso yield $\htau_1^2=X_1^T(X_1-X_{-1}\shgamma_1)/n$. Then, $\htheta_1$ is given by 
\begin{equation*}
    \htheta_1:=\frac{1}{\htau_1^2}(1,-\shgamma_1^T)^T.
\end{equation*}
By a simple calculation, we get
\begin{equation}
    \sqrt{n}(\sbhat_1-\beta_{01})=\frac{1}{\sqrt{n}}\htheta_1^TX^T\epsilon+\Delta_1, \label{140}
\end{equation}
where 
\begin{equation*}
    \Delta_1:=\sqrt{n}(\hsigma\htheta_1-e_1)^T(\beta_0-\shbeta),
\end{equation*}
and $e_1$ is the unit vector whose first element is unity. \par

From \eqref{140}, the proof of asymptotic normality hinges on showing that the bias term $\Delta_1$ is asymptotically negligible.
For $\min\{s_0, s_1\}= o(\sqrt{n}/ \log p)$, \cite{jm2018} and \cite{bz2022} developed skillful techniques to evaluate the bias. 
Otherwise, it is common to rely on the $\ell_1$-$\ell_\infty$ H\"{o}lder inequality:
\begin{equation}
    |\Delta_1| \leq \sqrt{n}\|\hsigma \htheta_1-e_1\|_\infty\|\shbeta-\beta_0\|_1\leq \sqrt{n}\frac{\lambda_1}{\htau_1^2}\|\shbeta-\beta_0\|_1. \label{148}
\end{equation}
The last inequality in \eqref{148} is due to the KKT conditions for the node-wise Lasso. 
Under certain conditions, the Lasso satisfies $\| \hat{\beta}-\beta_0 \|_1=O_p(s_0 \sqrt{\log p/n})$ and the convergence rate cannot be improved.
Thus, the order of the bias significantly depends on $\lambda_1/\hat{\tau}_1^2$.
It has been shown that $\lambda_1/\hat{\tau}_1^2 =O_p(\sqrt{\log p/n})$ for the choice of $\lambda_1 \asymp \sqrt{\log p/n}$ in previous studies, such as \cite{vbrd2014}.
They have shown that $1/\hat{\tau}_1^2=O_p(1)$ using the fact that $\hat{\gamma}_1$ converges to its population counterpart if $\lambda_1 \asymp \sqrt{\log p/n}$ and $s_1 =o(n/\log p)$.\par

In this study, we present a new evaluation method of the bias by showing that $1/ \hat{\tau}_1^2=O_p(1)$ for $\lambda_1 \asymp 1/\sqrt{n}$.
Different from previous studies, we do not need consistency of the node-wise Lasso. 
Instead, we bound $1/\htau_1^2$ using the properties of the node-wise Lasso derived from the general position assumption and conditions on the restricted eigenvalues of $\hat{\Sigma}$. 
This allows us not to restrict the sparsity of $\Theta_1$. 
A crucial point to bound $1/\htau_1^2$  is bounding  $1/\ttau_1^2$, which is an upper bound of  $1/\htau_1^2$ and the variance factor of the debiased Lasso. 
Because \cite{zz2014} showed that $1/\ttau_1^2$ is a nonincreasing function in $\lambda_1$, selecting too small $\lambda_1$ may make $1/\ttau_1^2$ and the variance diverge. 
We show that the debiased Lasso is $\sqrt{n}$-consistent even for $\lambda_1 \asymp 1/\sqrt{n}$, although the asymptotic variance may be greater than when $\lambda_1 \asymp \sqrt{\log p/n}$.

\par

\subsection{Asymptotic Normality with Small Tuning Parameter}
To obtain the new bound for the bias, we first introduce four generic assumptions.
Then, we show that the assumptions are satisfied under rather simple conditions.
Roughly speaking, our assumptions are satisfied if $p$ is bounded by a constant multiple of $n$.

The first assumption is the general position assumption of \cite{t2013}.
 \begin{as}
 \label{a1}
 \textup{For any $k<\min \{n,p\}$, the affine span of $\sigma_{1} X_{i_1},\ldots, \sigma_{k+1}X_{i_{k+1}}$ for arbitrary signs $\sigma_1,\ldots,\sigma_{i_{k+1}} \in \{-1,1\}$ does not include any element of $\{\pm X_i: i\neq i_1,\ldots,i_{k+1}\}.$}
 \end{as}
\cite{t2013} showed that the Lasso is unique and selects at most $\min\{n,p\}$ covariates if each column vector of $X$ is in general position. 
This assumption also implies that the covariates selected by the Lasso are linearly independent with probability one.
 A sufficient condition for the general position assumption is that the distribution of each element in $X$ is absolutely continuous with respect to the Lebesgue measure on $\mathbb{R}^{np}$. See Lemma 4 in \cite{t2013}. \par

Before introducing the next assumption, we define the restricted maximum and minimum eigenvalues of the sample covariance of $X$:
 \begin{equation*}
    \phi_{\max}(s):=\max_{1 \leq \|z\|_0\leq s}\frac{\|Xz\|^2}{n\|z\|^2},\ \ \phi_{\min}(s):=\min_{1 \leq \|z\|_0\leq s}\frac{\|Xz\|^2}{n\|z\|^2}.
\end{equation*}
These two quantities also play crucial roles in evaluating $1/\ttau_1^2.$ \par

\begin{as}
\label{a2}
\textup{There are a positive constant $C_{\min}$ and a positive constant $K$ such that
\begin{equation*}
   \lim_{n \rightarrow \infty}P\left(\phi_{\min}\left(\frac{n}{K}\right)\geq C_{\min}\right)=1.
\end{equation*}
}
\end{as}
We assume that $n$ is divisible by $K$ to simplify the notation. 
Under this assumption, the minimum eigenvalues of all sub-sample covariance matrices formed by less than $n/K$ covariates are bounded below with probability tending to one. 
Assumption \ref{a1} only tells us that the node-wise Lasso selects less than or equal to $n$ covariates.
Assumption \ref{a2} plays the role of bounding $1/\ttau_1^2$ above when the node-wise Lasso selects less than $n/K$ covariates. \par

\begin{as}
\label{a3}
\textup{There is a positive constant $C_{\max}$ such that 
\begin{equation*}
    \lim_{n \rightarrow \infty}P(\phi_{\max}(n)\leq C_{\max})=1.
\end{equation*}
}
\end{as}
Under this assumption, the maximum eigenvalues of all sub-sample covariance matrices formed by no greater than $n$ covariates are bounded above with probability tending to one. This assumption is to bound $1/\ttau_1^2$ above when the node-wise Lasso selects at least $n/K$ covariates.  \par

Under Assumptions 1--3, we can bound $1/\ttau_1^2$ above even when $\lambda_1 \asymp 1/\sqrt{n}$.
\begin{lem}\label{lem1}
Suppose that Assumptions \ref{a1}--\ref{a3} hold and let $\lambda_1$ satisfy $\lambda_1 \ge C_1 /\sqrt{n}$ for some $C_1 >0$. Then, we have
\begin{equation*}
    \lim_{n\rightarrow \infty}P\left(\frac{1}{\ttau_1^2}\leq \frac{1}{C_{\min}}+\frac{C_{\max}K}{C_1^2}\right )=1.
\end{equation*}
\end{lem}

Lemma \ref{lem1} implies that the convergence rate of the bias is improved over previous studies.
\cite{vbrd2014} and \cite{jm2018} obtain $\lambda_1/ \hat{\tau}_1^2=O_p(\sqrt{\log p/n})$ under the assumption that $s_1=o_p(n/\log p)$. 
\cite{v2019} does not assume the sparsity of $\Theta_1$ but obtains the same convergence rate of the bias.
Meanwhile, we have $\lambda_1/\hat{\tau}_1^2=O_P(1/\sqrt{n})$, and our derivation does not need any sparsity on $\Theta_1$.
Therefore, if the Lasso satisfies $\| \hat{\beta} - \beta_0 \|_1=O_p(s_0 \sqrt{\log p/n})$,  zero-mean asymptotic normality is satisfied for $s_0=o(\sqrt{n/\log p})$.
Lemma 1 also indicates that if we allow $1/C_{\min}$, $C_{\max}$, or $K$ to diverge, the order of $\lambda_1/\hat{\tau}_1^2$ will be greater than $O_P(1/\sqrt{n})$ and $s_0$ should be less than $o(\sqrt{n/\log p})$ for asymptotic normality.
\par 
Before we formally state the main result, we define the restricted eigenvalue of the sample covariance introduced in \cite{brt2009}: 
\begin{equation*}
    \kappa(s,c_0)^2:=\min_{\substack{S\subset\{1,\ldots,p\}\\ |S|\leq s}}\min_{\substack{z\neq 0 \\ \|z_{S^c}\|_1\leq c_0\|z_{S}\|_1}}\frac{\|Xz\|^2}{n\|z_{S}\|^2}.
\end{equation*} 
We put the following assumptions on $\phi_{\max}(1)$ and the restricted eigenvalue.
\begin{as}
\label{a4}
\textup{(i) There exists a positive constant $C^*$ such that 
\begin{equation*}
    P(\phi_{\max}(1)>C^*)=O\left(\frac{1}{p}\right).
\end{equation*} (ii) For $s_0=o(\sqrt{n/\log p})$, there exists a positive constant $\kappa_{\min}$ such that}
\begin{equation*}
    \lim_{n\rightarrow \infty }P(\kappa(s_0,3)^2\geq \kappa_{\min}^2)=1.
\end{equation*}
\end{as}
We use these assumptions to derive the oracle inequalities of the Lasso, which are required to evaluate the bias $\Delta_1$. The first assumption holds, for instance, if each row vector of $X$ is independent Gaussian.
Further imposing $n \gtrsim s_0 \log (p/s_0)$, the second assumption is also satisfied.
See Theorem 6 of \cite{rz2013}.

\par

Now, we present a result of asymptotic normality for the debiased Lasso when $\lambda_1 \asymp 1/\sqrt{n}$.

\begin{thm}\label{thm1}
Suppose Assumptions \ref{a1}--\ref{a4} hold and $s_0$ satisfies $s_0=o(\sqrt{n/\log p})$.  
Then, for suitably selected $\lambda_0 \asymp\sqrt{\log p/n}$ and $\lambda_1 \asymp 1/\sqrt{n}$,  we have 
\begin{equation*}
    \sqrt{n}(\sbhat_1-\beta_{01})=W_1+\Delta_1,
\end{equation*}
\begin{equation*}
    W_1|X \sim N(0,\sigma_\epsilon^2 \homega_{1,1}),
\end{equation*}
\begin{equation*}
    \homega_{1,1}:=\htheta_1{\! \!}^ T\hsigma\htheta_1=O_P(1),
\end{equation*}
\begin{equation*}
    \Delta_1=o_P(1).
\end{equation*}
\end{thm}

Our variance of the debiased Lasso has the same form as that of \cite{vbrd2014}. 
However, our variance is typically larger than theirs because our $\hat{\Theta}_1$ may not be sparse due to small $\lambda_1$.
Moreover, $\hat{\Omega}_{1,1}$ may not converge in probability to the $(1,1)$ component of $\Sigma^{-1}$.

We can also establish asymptotic normality for the studentized estimator.
\begin{cor}\label{cor1}
Suppose that the assumption of Theorem 1 holds. 
Let $\shsigma_\epsilon^2$ be a consistent estimator for $\sigma_\epsilon^2$. Then, we have 
\begin{equation*}
\frac{\sqrt{n}(\sbhat_1-\beta_{01})}{\shsigma_\epsilon\homega_{1,1}^{1/2}}\xrightarrow{d}N(0,1).
\end{equation*}
\end{cor}

There are some known consistent estimators for $\sigma_\epsilon^2$. 
For example, we can set $\shsigma_\epsilon^2=\|Y-X\hbeta\|^2/n$ with $\lambda_0\ \asymp \sqrt{\log p/n}$ if $s_0=o(n/\log p).$  
We can also use the scaled Lasso proposed by \cite{sz2012} to obtain a consistent estimate of $\sigma_\epsilon^2$.

Corollary \ref{cor1} states that even if $\sqrt{n}/\log p \lesssim s_0 \ll \sqrt{n/\log p}$ and $\Theta_1$ is not sparse, the studentized debiased Lasso is still valid to perform the $t$ test or construct confidence intervals.
We can confirm that the length of confidence intervals given by the debiased Lasso with $\lambda_1 \asymp 1/\sqrt{n}$ is $O_P(1/\sqrt{n})$ because variances are bounded. \par

Finally, we give a set of simple conditions under which Assumptions 1--4 hold.
\begin{prop}
	\label{prop1}
	Suppose that each row vector of  $X$ is independently drawn from $N(0, \Sigma)$ where $\Sigma$ satisfies $K_{\min}\leq \lambda_{\min}(\Sigma)\leq \lambda_{\max}(\Sigma)\leq K_{\max}$ for some positive constants $K_{\min}$ and $K_{\max}$.
	Moreover, we assume that $s_0 = o(n/\log p)$ and $p \leq C_0n$ for some positive constant $C_0 < \infty$.
 Then, Assumptions 1--4 hold. 
\end{prop}

Without the condition on $p$, we only have $\phi_{\max}(n) \lesssim \sqrt{\log p}$ and $\phi_{\min}(s_n) \gtrsim 1$ for $s_n=o(n/\log p)$ with high probability (see, for instance, the proof of Lemma 3.5 of \citealt{jm2018}).
The condition on $p$ is used to give a tighter upper bound for $\phi_{\max}(n)$ and lower bound for $\phi_{\min}(n/K)$.

\begin{rem}
	\textup{
		It has not been shown that other inference methods, such as the post-double selection and double machine learning, can have asymptotic normality under similar conditions to ours. 
		Original proof techniques of these methods depend on convergence rates of machine learning estimators, but these rates are unavailable if no sparsity is assumed on $\Theta_1$.}
\end{rem}
\par

\begin{rem}
	\textup{\cite{rszz2015}, \cite{cg2017}, and \cite{jm2018} showed that confidence intervals with length of order $1/\sqrt{n}$ cannot be constructed when $s_0 \gtrsim \sqrt{n}/\log p$ and the population covariance matrix is unknown. 
	However, our results do not contradict their results because they only considered the case where $p>s_0^{\alpha}$ for some $\alpha>2$. 
	If $p \leq Cn$, we have $p/s_0^{\alpha} \leq Cn^{1-\alpha/2} (\log p)^{\alpha/2} \rightarrow 0$ for $s_0$ such that $s_0 \lesssim  \sqrt{n/\log p}$. 
	Therefore, we obtain $p\leq s_0^{\alpha}$ for large $n$ and their condition does not hold. }
\end{rem}

\section{Tuning Parameter Selection Procedure}

From Theorem 1 and Corollary 1, it is better to select a small tuning parameter such that $\lambda_1 \asymp 1/\sqrt{n}$ for statistical inference on $\beta_{01}$ if you are unsure of how sparse $\beta_0$ and $\Theta_1$ are. 
However, our theoretical results do not indicate how to select specific $\lambda_1$ in practice. 
Existing selection procedures may not select a sufficiently small tuning parameter needed to perform accurate statistical inference. 
For example, the cross-validation method is designed to minimize the prediction mean squared error, so it may be unsuitable for statistical inference.
The method of \cite{dbmm2015} selects a slightly smaller value than the cross-validation method, but it is not theoretically justified.
Therefore, we propose a tuning parameter selection procedure for the node-wise Lasso that yields stable confidence intervals.

Now, we investigate the bias term of the studentized debiased Lasso.
Let $\Lambda$ be a set of candidate values of $\lambda_1$.
For each $\lambda \in \Lambda$, we define
\begin{equation*}
    f(\lambda):=\frac{\|\hsigma\htheta_1(\lambda)-e_1\|_\infty}{\homega_{1,1}^{1/2}(\lambda)},
\end{equation*}
where $\htheta_1(\lambda)$ and $\homega_{1,1}(\lambda)$ denote $\htheta_1$ and $\homega_{1,1}$ when $\lambda_1 =\lambda$, respectively.
The $\ell_1$-$\ell_\infty$ H\"{o}lder inequality for the bias term of the studentized debiased Lasso $\tilde{\Delta}_1$ yields 
\begin{equation*}
    |\tilde{\Delta}_1(\lambda)|\leq   f(\lambda)\sqrt{n}\shsigma_\epsilon^{-1}\|\hbeta-\beta_0\|_1. 
\end{equation*}
We can trace the bias factor $f(\lambda)$ for each $\lambda \in \Lambda$ because $\htheta_1$ and $\homega_{1,1}$ only depend on $X$. 
Note that $\sqrt{n}\shsigma_\epsilon^{-1}\|\hbeta-\beta_0\|_1$ is independent of the tuning parameter of the node-wise Lasso. 
If $f(\lambda)$ is small, $\tilde{\Delta}_1(\lambda)$ is also small, so that the distribution of the studentized debiased Lasso is well-approximated by the standard normal distribution. 
However, minimizing $f(\lambda)$ is inappropriate because it is nondecreasing in $\lambda$ (see Proposition 1 of \citealt{zz2014}).  
Too small tuning parameter unnecessarily increases the variance of the debiased Lasso.

To select a suitably small tuning parameter for the node-wise Lasso, we propose the following method, a simple and conservative modification of the method described in Table 2 of \cite{zz2014}.
\begin{enumerate}
    \item Let $\Lambda=\{\lambda_{\min},\ldots,\lambda_{\max}\}$ be a candidate set of $\lambda_1$ with $\lambda_{\min}=\kappa/\sqrt{n}$ for some $\kappa>0$ and $\lambda_{\max}=\max_{k \neq 1}|X_k^TX_1|/n$. Let
    \begin{equation*}
        \eta_1^*=\frac{\ttau_1(\lambda_{\text{CV}})}{\sqrt{n}},
    \end{equation*}
   where $\ttau_1^2(\lambda_{\text{CV}})$ denotes the sample mean of squared residuals from the cross-validated node-wise Lasso.
   
    \item Select $\lambda_1$ with the following rule:
    \begin{enumerate}
        \item If $f(\lambda)>\eta_1^*$ for all $\lambda \in \Lambda$, then set $\lambda_1=\lambda_{\min}$.
        \item Otherwise,
        \begin{align*}
            \lambda^*&=\argmin\{\homega_{1,1}^{1/2}(\lambda): f(\lambda) \leq \eta_1^*\}, \\
            \homega_{1,1}^{*1/2}&=1.1\times\homega_{1,1}^{1/2}(\lambda^*), \\
            \lambda_1&=\argmin\{f(\lambda):\homega_{1,1}^{1/2}(\lambda)\leq \homega_{1,1}^{*1/2}\}.
        \end{align*}
    \end{enumerate}
    \label{node} 
\end{enumerate}
We call this proposed method STPS. 
The idea of STPS is simple. 
It minimizes the bias with a constraint on the upper bound of the variance.
When $f(\lambda) \leq \eta_1^*$ for some $\lambda \in \Lambda$, $\lambda_1$ is also equal to the smallest value in $\{\lambda: \homega_{1,1}^{1/2}(\lambda)\leq \homega_{1,1}^{*1/2}\}$. 
Simulation studies show that $\kappa=0.001$ is an appropriate value to obtain accurate coverage.

The main difference between our method and the method of \cite{zz2014} is that our upper bound of $f(\lambda)$ is much smaller than theirs, which is given by $\eta_1^*=\sqrt{2\log p/n}$. 
Therefore, our method can select a smaller tuning parameter than theirs, and we can show that $f(\lambda_1)=O_P(1/\sqrt{n})$. 
Moreover, STPS does not increase the standard error overly because we can show that  $\homega_{1,1}^{*1/2}$ is bounded. 
Meanwhile, if we select $\lambda_1$ using their method, the standard error may be smaller than ours; however, we obtain $f(\lambda_1)=O_P(\sqrt{\log p/n})$, meaning that the bias may be large and coverage of confidence intervals may be small if $\beta_0$ and $\Theta_1$ are not sufficiently sparse. 
\par 
Finally, we present an asymptotic justification for the proposed procedure.
\begin{thm}\label{thm2}
Suppose that the assumption of Theorem 1 holds, and $\lambda_1$ is selected by STPS. Then, we have
\begin{equation*}
    \frac{\sqrt{n}(\sbhat_1-\beta_{01})}{\shsigma_\epsilon \homega_{1,1}^{1/2}}\xrightarrow{d}N(0,1).
\end{equation*}
\end{thm}

\begin{rem}
	\textup{For our selection method to work well, $\Lambda$ must contain small candidate values.
		However, well-known software packages typically do not generate sufficiently small values as candidates.
		Thus, it is necessary to manually include small values to $\Lambda$.
		An example of the choice of $\Lambda$ is illustrated in the next section.}
\end{rem}

\section{Simulation Studies}
In this section, we compare finite sample performances of the debiased Lasso tuned by our procedure and those of other statistical inference methods. 
We use the R packages \verb|glmnet| and \verb|hdm|. In all examples, we set $p=2n$ and repeat experiments 1,000 times.

\subsection{Simulation Design}
 We employ the settings of \cite{whx2020}. 
 For a given value of $c_y$, we set $\beta_0$ as one of the following $p+1$-dimensional vectors.
\begin{align*}
    \text{sparse}&:(1.5, \underbrace{c_y,\ldots, c_y}_{5},0,\ldots,0)^T, \\
    \text{moderately sparse}&:(1.5, \underbrace{5c_y,\ldots,5c_y}_{10},\underbrace{c_y,\ldots,c_y}_{10},0,\ldots,0)^T, \\
    \text{dense}&: (1.5, c_y, c_y/\sqrt{2},\ldots, c_y/\sqrt{p})^T.
\end{align*}
In addition, for a given value of $c_x$, we set $\gamma_0$ as one of the following $p$-dimensional vectors.
\begin{align*}
    &\text{sparse}:(\underbrace{0,\ldots,0}_{4}, \underbrace{c_x,\ldots,c_x}_{4},0,\ldots,0)^T, \\
    &\text{dense}: (c_x, c_x/\sqrt{2},\ldots, c_x/\sqrt{p})^T.
\end{align*}
The covariance matrix $\Sigma$ is given by one of the following matrices:
\begin{align*}
    &\text{Independent}: \Sigma=I,\ \ \text{Toeplitz}: \Sigma_{jk}=\rho^{|j-k|}, \\
    &\text{Equal correlation}: \Sigma_{jk}=\rho^{1(j \neq k)},\ \ \rho \in \{0.3, 0.9\}.
\end{align*}

We consider the following two settings.
\begin{itemize}
    \item Setting 1. The data generating process is given by
    \begin{equation*}
        Y_i=1+ X_i^T\beta_0 +\epsilon_i \quad \text{for $i=1, \dots, n$} ,
    \end{equation*}
 where  $X_i \in \mathbb{R}^{p+1}  \sim N(0, \Sigma)$, $\epsilon_i \sim N(0,1)$, and $X_i$ and $\epsilon_i$ are independent. 
    We select $c_y$ so that $R^2$ is 0.8.
    
    \item Setting 2. The data generating process is given by
    \begin{align*}
        Y_i &=1 + X_i^T \beta_0 +\epsilon_i, \\
        X_{i,1}&=0.5+X_{i,-1}^T\gamma_0+\eta_i,
    \end{align*}
    where $X_{i,-1} \in \mathbb{R}^{p} \sim N(0, \Sigma),$ $(\epsilon_i,\eta_i)\sim N(0, I_2),$ and $X_i=(X_{i,1}, X_{i,-1}^T)^T$ and $(\epsilon_i, \eta_i)$ are independent. 
    We select $c_y$ so that $R^2$ of the first equation is 0.8 and select $c_x$ so that $R^2$ of the second equation is one of \{0.3, 0.5, 0.8, 0.9\}.
\end{itemize}

The true value of the parameter of interest, the first element of $\beta_0$, is  1.5. 
In Setting 1, some calculations yield
\begin{equation*}
    \tau_1^2=\begin{cases}
    1 & \text{if} \ \ \  \Sigma_{jk}=1(j=k), \\
    0.701 & \text{if} \ \ \  \Sigma_{jk}=0.3^{1(j \neq k)}, \\
    0.19 & \text{if}\ \ \ \Sigma_{jk}=0.9^{|j-k|}, \\
    0.1 & \text{if}\ \ \  \Sigma_{jk}=0.9^{1(j \neq k)}.
    \end{cases}
\end{equation*}
In Setting 2, we have $\tau_1^2=1$.
When the covariance matrix has the structure of the Toeplitz or equal correlation, $\Theta_1$ is dense for Setting 1. 
In Setting 2, $\gamma_0$ controls the sparsity of $\Theta_1$ because we have $\Theta_1=(1,-\gamma_0^T)^T$.

The following procedures are included in the comparisons.
\begin{itemize}
	\item ``Oracle" refers to OLS derived from true models. 
    \item ``STPS" refers to the debiased Lasso with $\lambda_1$ selected by STPS.
	\item ``CV" refers to the debiased Lasso with $\lambda_1$ selected by 10-fold cross-validation. 
	\item``Univ" refers to the debiased Lasso with $\lambda_1=\tau_1\sqrt{2 \log p/n}$, the universal tuning parameter.
	\item ``DBMM" refers to the debiased Lasso with $\lambda_1$ selected by the method of \cite{dbmm2015}. 
	\item ``ZZ" represents the debiased Lasso with $\lambda_1$ selected by the method of \cite{zz2014}. 
	\item ``Double" refers to the post-double selection of \cite{bch2014}.
    \item ``PODS" refers to the projection onto the double selection of \cite{whx2020}.
	\item ``DML" refers to the double machine learning of \cite{ccddhnr2018} with 5-fold sample splitting. 
\end{itemize}
We compared the biases, standard deviations, coverages, and lengths of 95\% confidence intervals for all methods.

To implement the debiased Lasso, we need to determine two tuning parameters.
For all debiased Lasso-based methods, including STPS, the tuning parameter of the Lasso is determined by the 10-fold cross-validation, which is implemented using the \verb|lambda.min| option of \verb|glmnet|.
The tuning parameter of the node-wise Lasso is selected from a common set $\Lambda$.
As we note in the previous section, the default set generated by \verb|glmnet| does not contain sufficiently small values.
Thus, we augment the default set by adding some small values.
Specifically, we set $\Lambda=\Lambda_1 \cup \Lambda_2$ where $\Lambda_1=\{0.001/\sqrt{n}, 0.002/\sqrt{n},\ldots, 0.02/\sqrt{n}\}$ consists of 20 elements, and $\Lambda_2$ is the set of 100 elements generated by $\verb|glmnet|$.
Note that we do not multiply $\lambda_0$ and $\lambda_1$ by 2 in our simulation studies.
 We estimate the error variance $\sigma_\epsilon^2=1$ by $\shsigma_\epsilon^2=\|Y-X\shbeta(\lambda_{\text{1se}})\|^2/n$, where $\shbeta(\lambda_{\text{1se}})$ is the Lasso whose tuning parameter is selected by the one standard error rule with \verb|lambda.1se| option.

We briefly explain other estimation methods. 
Double is implemented by the function \verb|rlassoEffects| of the R package \verb|hdm|. In PODS, we use the Lasso implemented by the function \verb|rlasso| of the R package \verb|hdm| as the model selection procedure.
To implement DML, we follow Definition 3.2 of \cite{ccddhnr2018} and use the variance estimator defined in Theorem 3.2 to construct confidence intervals. 
The orthogonal score function $\psi$ is defined as follows: 
\begin{equation*}
\psi(W_i,\eta_0,\beta_{01})=(Y_i-X_{i,1}\beta_{01}-g_0(X_{i,-1}))(X_{i,1}-m_0(X_{i,-1})),
\end{equation*}
where $\eta_0=(g_0,m_0)$ is a nuisance parameter and $W_i=(X_i,Y_i)$. 
We estimate $g_0$ and $m_0$ by the post-Lasso, which is implemented by the function \verb|rlasso|. 
Tuning parameters for the post-Lasso are selected using the method of \cite{bcch2012}, which is also implemented by \verb|rlasso|.

\subsection{Simulation Results}

Table \ref{357} summarizes the results for Setting 1 with $n=100$ and $p=200$. 
In all cases, the coverage rates of STPS are close to the nominal level. 
On the other hand, the debiased Lasso-based confidence intervals have smaller coverage rates if $\lambda_1$ is selected by other methods. 
This is because they select relatively larger tuning parameters so that the debiased Lasso tends to have a large bias and small variance.
In particular, the universal penalty, which is often used for theoretical analysis of the Lasso, may not be suitable for the node-wise Lasso. 
CV also yields low coverage, although it is the most well-known method for tuning parameter selection. 
Double and PODS are also unstable, and their coverages are generally low.  
If the correlation among covariates is not too high, DML performs well and its standard deviations are smaller than those of STPS. 
However, when the correlation among covariates is the highest, the coverage of DML is much higher than the nominal level. \par

Table \ref{413} shows the results for Setting 1 when $n=500$ and $p=1000$. 
STPS gives confidence intervals with accurate coverage.
Even in this case, Univ still has large biases, and its coverage is small. 
Although the results of CV, DBMM, and ZZ are improved compared with the small sample size case, their coverage is still small, particularly when $\beta_0$ is dense and $\Sigma$ has the Toeplitz structure. 
DML performs well, but its coverage is still large when the correlation is the highest. 
The performances of Double and PODS are unstable, even in the large sample case. \par

Next, we compare the results of all procedures for Setting 2. 
Table \ref{497} shows the simulation results for this setting when $n=100$ and $p=200$. 
Only the coverage rates of STPS are close to the nominal coverage in all cases. 
Other debiased Lasso-based methods do not cover $\beta_{01}$ well in this setting. 
In particular, when $\beta_0$ and $\gamma_0$ are dense, they have large biases. 
These results are consistent with our theoretical analysis. 
Because the sparsity assumptions of \cite{vbrd2014} and \cite{jm2018} are violated when both $\beta_0$ and $\gamma_0$ are dense, existing selection methods cannot remove the bias.
Although DML performs well when either $\beta_0$ or $\gamma_0$ is not dense, its performance is significantly aggravated when both of them are dense. 
The theory of DML allows that one nuisance parameter is dense to make the bias small if the other nuisance parameter is quite sparse. However, because both nuisance parameters are dense, a large bias occurs for DML. \par

Table \ref{498} shows simulation results for Setting 2 when $n=500$ and $p=1,000$. 
When $\beta_0$ and $\gamma_0$ are dense, the biases for all methods other than STPS are so large that confidence intervals yield much small coverage. 
This is because the other methods miss more relevant covariates than the case of $n=100$. \par

We also performed simulation for $p > 2n$.  The results are similar to the case of $p=2n$ and are not reported here.
We confirmed that STPS outperforms other methods in terms of coverage for $p \le 10 n$. 
However, when $p=10n$, the bias of the debiased Lasso cannot be sufficiently removed, even by selecting the smallest value in $\Lambda$. 
Therefore, the coverage of STPS also tends to be smaller than the nominal level.\par

In summary, provided $p$ is not significantly greater than $n$, STPS gives accurate confidence intervals regardless of whether $\beta_0$ and $\Theta_1$ are sparse or dense. 
Meanwhile, other methods perform poorly, particularly when $\beta_0$ and $\Theta_1$ are dense. 
Our simulation results indicate that the debiased Lasso tuned by our method would enable us to perform accurate statistical inference in general settings when $p$ is reasonably greater than $n$. \par 
    
\begin{table}
\vspace{-2em}
\centering
\caption{Simulation results for Setting 1 with $n=100$ and $p=200$.}
{\small
\begin{tabular}{cccccccccc}
\hline
       & Oracle & STPS   & Univ   & CV     & DBMM   & ZZ     & Double & PODS   & DML    \\ \hline
       & \multicolumn{9}{c}{$\beta_0$ is sparse, $\Sigma=I$}                                \\ \cline{2-10} 
Bias   & 0.001  & -0.039 & -0.052 & -0.051 & -0.045 & -0.043 & -0.162 & -0.006 & -0.003 \\
SD     & 0.103  & 0.165  & 0.121  & 0.122  & 0.13   & 0.139  & 0.161  & 0.163  & 0.127  \\
Cover  & 0.932  & 0.94   & 0.875  & 0.886  & 0.909  & 0.919  & 0.757  & 0.831  & 0.938  \\
Length & 0.392  & 0.671  & 0.427  & 0.435  & 0.485  & 0.532  & 0.545  & 0.442  & 0.462  \\ \cline{2-10} 
       & \multicolumn{9}{c}{$\beta_0$ is sparse, $\Sigma_{j,k}=0.3^{1(j\neq k)}$}                              \\ \cline{2-10} 
Bias   & -0.006 & -0.026 & -0.046 & -0.042 & -0.038 & -0.036 & -0.041 & 0.019  & 0.006  \\
SD     & 0.111  & 0.192  & 0.133  & 0.139  & 0.145  & 0.145  & 0.159  & 0.213  & 0.153  \\
Cover  & 0.949  & 0.944  & 0.867  & 0.899  & 0.921  & 0.919  & 0.904  & 0.775  & 0.95   \\
Length & 0.434  & 0.757  & 0.424  & 0.482  & 0.537  & 0.542  & 0.569  & 0.511  & 0.597  \\ \cline{2-10} 
       & \multicolumn{9}{c}{$\beta_0$ is sparse, $\Sigma_{j,k}=0.9^{1(j\neq k)}$}                              \\ \cline{2-10} 
Bias   & 0.002  & -0.067 & -0.311 & -0.114 & -0.087 & -0.218 & 0.331  & 0.014  & 0.14   \\
SD     & 0.291  & 0.489  & 0.325  & 0.347  & 0.357  & 0.332  & 0.367  & 0.566  & 0.387  \\
Cover  & 0.944  & 0.935  & 0.509  & 0.886  & 0.917  & 0.739  & 0.848  & 0.732  & 0.984  \\
Length & 1.109  & 1.841  & 0.652  & 1.201  & 1.309  & 0.894  & 1.425  & 1.28   & 2.037  \\ \cline{2-10} 
       & \multicolumn{9}{c}{$\beta_0$ is sparse, $\Sigma_{j,k}=0.9^{|j-k|}$}                               \\ \cline{2-10} 
Bias   & 0.001  & -0.011 & 0.014  & 0.009  & 0.004  & 0.005  & -0.064 & -0.004 & -0.001 \\
SD     & 0.24   & 0.384  & 0.224  & 0.234  & 0.241  & 0.232  & 0.267  & 0.284  & 0.25   \\
Cover  & 0.934  & 0.945  & 0.826  & 0.886  & 0.913  & 0.901  & 0.908  & 0.905  & 0.932  \\
Length & 0.9    & 1.525  & 0.601  & 0.734  & 0.814  & 0.749  & 0.953  & 0.932  & 0.931  \\ \cline{2-10} 
       & \multicolumn{9}{c}{$\beta_0$ is moderately sparse, $\Sigma=I$}                     \\ \cline{2-10} 
Bias   & 0      & -0.066 & -0.087 & -0.086 & -0.077 & -0.073 & -0.265 & -0.009 & -0.002 \\
SD     & 0.111  & 0.179  & 0.147  & 0.148  & 0.152  & 0.158  & 0.186  & 0.183  & 0.173  \\
Cover  & 0.92   & 0.909  & 0.79   & 0.797  & 0.833  & 0.86   & 0.622  & 0.834  & 0.926  \\
Length & 0.392  & 0.7    & 0.446  & 0.454  & 0.506  & 0.555  & 0.653  & 0.496  & 0.632  \\ \cline{2-10} 
       & \multicolumn{9}{c}{$\beta_0$ is moderately sparse, $\Sigma_{j,k}=0.9^{1(j\neq k)}$}                   \\ \cline{2-10} 
Bias   & 0.002  & -0.066 & -0.31  & -0.113 & -0.085 & -0.217 & 0.333  & 0.014  & 0.138  \\
SD     & 0.348  & 0.488  & 0.325  & 0.346  & 0.356  & 0.331  & 0.369  & 0.565  & 0.386  \\
Cover  & 0.908  & 0.938  & 0.51   & 0.885  & 0.92   & 0.743  & 0.845  & 0.732  & 0.986  \\
Length & 1.203  & 1.837  & 0.651  & 1.199  & 1.306  & 0.893  & 1.424  & 1.277  & 2.028  \\ \cline{2-10} 
       & \multicolumn{9}{c}{$\beta_0$ is moderately sparse, $\Sigma_{j,k}=0.9^{|j-k|}$}                    \\ \cline{2-10} 
Bias   & 0.003  & -0.015 & -0.008 & -0.009 & -0.012 & -0.013 & -0.077 & -0.003 & 0.002  \\
SD     & 0.265  & 0.383  & 0.214  & 0.227  & 0.236  & 0.226  & 0.267  & 0.288  & 0.253  \\
Cover  & 0.913  & 0.949  & 0.835  & 0.893  & 0.911  & 0.898  & 0.899  & 0.904  & 0.934  \\
Length & 0.899  & 1.515  & 0.597  & 0.729  & 0.809  & 0.744  & 0.951  & 0.933  & 0.951  \\ \cline{2-10} 
       & \multicolumn{9}{c}{$\beta_0$ is dense, $\Sigma_{j,k}=0.9^{1(j\neq k)}$}                               \\ \cline{2-10} 
Bias   & -  & -0.065 & -0.31  & -0.112 & -0.085 & -0.216 & 0.334  & 0.014  & 0.138  \\
SD     & -  & 0.488  & 0.325  & 0.346  & 0.357  & 0.331  & 0.369  & 0.567  & 0.387  \\
Cover  & -  & 0.939  & 0.516  & 0.891  & 0.92   & 0.74   & 0.85   & 0.728  & 0.985  \\
Length & -  & 1.836  & 0.65   & 1.198  & 1.306  & 0.892  & 1.423  & 1.278  & 2.028  \\ \cline{2-10} 
       & \multicolumn{9}{c}{$\beta_0$ is dense, $\Sigma_{j,k}=0.9^{|j-k|}$}                                \\ \cline{2-10} 
Bias   & -  & -0.021 & -0.037 & -0.032 & -0.03  & -0.034 & -0.152 & -0.013 & 0.002  \\
SD     & -  & 0.381  & 0.225  & 0.235  & 0.242  & 0.234  & 0.264  & 0.297  & 0.289  \\
Cover  & -  & 0.939  & 0.785  & 0.842  & 0.876  & 0.862  & 0.865  & 0.873  & 0.944  \\
Length & -  & 1.459  & 0.575  & 0.702  & 0.779  & 0.716  & 0.955  & 0.931  & 1.09   \\ \hline
\end{tabular}
}
\label{357}
\end{table}
\begin{table}
\vspace{-2em}
\centering
\caption{Simulation results for Setting 1 with $n=500$ and $p=1000$.}
{\small 
\begin{tabular}{cccccccccc}
\hline
       & Oracle & STPS   & Univ   & CV     & DBMM   & ZZ     & Double & PODS   & DML    \\ \hline
       & \multicolumn{9}{c}{$\beta_0$ is sparse, $\Sigma=I$}                                         \\ \cline{2-10} 
Bias   & -0.002 & -0.005 & -0.008 & -0.008 & -0.006 & -0.006 & -0.1   & -0.001 & -0.002 \\
SD     & 0.046  & 0.076  & 0.048  & 0.048  & 0.052  & 0.057  & 0.062  & 0.061  & 0.046  \\
Cover  & 0.94   & 0.964  & 0.945  & 0.943  & 0.957  & 0.953  & 0.536  & 0.862  & 0.942  \\
Length & 0.176  & 0.314  & 0.185  & 0.185  & 0.207  & 0.23   & 0.211  & 0.185  & 0.176  \\ \cline{2-10} 
       & \multicolumn{9}{c}{$\beta_0$ is sparse, $\Sigma_{j,k}=0.3^{1(j\neq k)}$}                              \\ \cline{2-10} 
Bias   & 0.001  & -0.004 & -0.01  & -0.008 & -0.006 & -0.006 & 0.03   & 0.02   & 0.021  \\
SD     & 0.049  & 0.096  & 0.054  & 0.055  & 0.059  & 0.063  & 0.062  & 0.078  & 0.057  \\
Cover  & 0.959  & 0.954  & 0.923  & 0.945  & 0.95   & 0.953  & 0.902  & 0.81   & 0.961  \\
Length & 0.195  & 0.379  & 0.196  & 0.211  & 0.235  & 0.248  & 0.229  & 0.214  & 0.242  \\ \cline{2-10} 
       & \multicolumn{9}{c}{$\beta_0$ is sparse, $\Sigma_{j,k}=0.9^{1(j\neq k)}$}                              \\ \cline{2-10} 
Bias   & 0.002  & -0.008 & -0.102 & -0.023 & -0.017 & -0.044 & 0.242  & 0.027  & 0.037  \\
SD     & 0.127  & 0.224  & 0.141  & 0.144  & 0.144  & 0.143  & 0.152  & 0.22   & 0.147  \\
Cover  & 0.95   & 0.953  & 0.721  & 0.938  & 0.947  & 0.908  & 0.636  & 0.786  & 0.975  \\
Length & 0.497  & 0.893  & 0.384  & 0.544  & 0.556  & 0.5    & 0.589  & 0.555  & 0.667  \\ \cline{2-10} 
       & \multicolumn{9}{c}{$\beta_0$ is sparse, $\Sigma_{j,k}=0.9^{|j-k|}$}                                \\ \cline{2-10} 
Bias   & 0      & 0.007  & 0.013  & 0.009  & 0.004  & 0.003  & -0.06  & -0.001 & 0      \\
SD     & 0.108  & 0.18   & 0.106  & 0.106  & 0.108  & 0.109  & 0.115  & 0.118  & 0.109  \\
Cover  & 0.941  & 0.958  & 0.85   & 0.89   & 0.92   & 0.925  & 0.892  & 0.924  & 0.941  \\
Length & 0.402  & 0.707  & 0.314  & 0.341  & 0.38   & 0.392  & 0.419  & 0.409  & 0.404  \\ \cline{2-10} 
       & \multicolumn{9}{c}{$\beta_0$ is moderately sparse, $\Sigma=I$}                      \\ \cline{2-10} 
Bias   & -0.002 & -0.008 & -0.014 & -0.014 & -0.011 & -0.01  & -0.106 & 0      & -0.002 \\
SD     & 0.047  & 0.077  & 0.05   & 0.05   & 0.054  & 0.058  & 0.063  & 0.063  & 0.048  \\
Cover  & 0.938  & 0.963  & 0.928  & 0.926  & 0.945  & 0.956  & 0.532  & 0.866  & 0.952  \\
Length & 0.176  & 0.323  & 0.19   & 0.19   & 0.212  & 0.236  & 0.22   & 0.19   & 0.185  \\ \cline{2-10} 
       & \multicolumn{9}{c}{$\beta_0$ is moderately sparse, $\Sigma_{j,k}=0.9^{1(j\neq k)}$}                   \\ \cline{2-10} 
Bias   & 0.003  & -0.008 & -0.101 & -0.022 & -0.017 & -0.043 & 0.242  & 0.026  & 0.037  \\
SD     & 0.138  & 0.224  & 0.14   & 0.144  & 0.144  & 0.143  & 0.151  & 0.22   & 0.147  \\
Cover  & 0.947  & 0.953  & 0.723  & 0.938  & 0.945  & 0.907  & 0.632  & 0.777  & 0.978  \\
Length & 0.541  & 0.891  & 0.383  & 0.543  & 0.555  & 0.499  & 0.587  & 0.554  & 0.665  \\ \cline{2-10} 
       & \multicolumn{9}{c}{$\beta_0$ is moderately sparse, $\Sigma_{j,k}=0.9^{|j-k|}$}                    \\ \cline{2-10} 
Bias   & 0      & 0.006  & 0.007  & 0.004  & 0      & 0      & -0.064 & -0.002 & 0      \\
SD     & 0.109  & 0.18   & 0.102  & 0.103  & 0.107  & 0.108  & 0.115  & 0.119  & 0.109  \\
Cover  & 0.94   & 0.959  & 0.871  & 0.896  & 0.924  & 0.928  & 0.892  & 0.919  & 0.943  \\
Length & 0.402  & 0.706  & 0.313  & 0.341  & 0.38   & 0.392  & 0.419  & 0.41   & 0.407  \\ \cline{2-10} 
       & \multicolumn{9}{c}{$\beta_0$ is dense, $\Sigma_{j,k}=0.9^{1(j\neq k)}$}                               \\ \cline{2-10} 
Bias   & -  & -0.008 & -0.101 & -0.022 & -0.016 & -0.043 & 0.241  & 0.024  & 0.038  \\
SD     & -  & 0.224  & 0.14   & 0.143  & 0.143  & 0.142  & 0.151  & 0.221  & 0.146  \\
Cover  & -  & 0.95   & 0.72   & 0.937  & 0.947  & 0.912  & 0.639  & 0.785  & 0.976  \\
Length & -  & 0.89   & 0.382  & 0.542  & 0.554  & 0.499  & 0.587  & 0.554  & 0.664  \\ \cline{2-10} 
       & \multicolumn{9}{c}{$\beta_0$ is dense, $\Sigma_{j,k}=0.9^{|j-k|}$}                                \\ \cline{2-10} 
Bias   & - & 0.004  & -0.009 & -0.01  & -0.01  & -0.009 & -0.16  & -0.003 & -0.002 \\
SD     & -  & 0.179  & 0.108  & 0.108  & 0.11   & 0.111  & 0.115  & 0.124  & 0.128  \\
Cover  & -   & 0.943  & 0.822  & 0.863  & 0.899  & 0.904  & 0.661  & 0.903  & 0.944  \\
Length & -  & 0.67   & 0.297  & 0.324  & 0.361  & 0.372  & 0.421  & 0.409  & 0.481  \\ \hline
\end{tabular}
}
\label{413}
\end{table}

\begin{table}
\centering
\caption{Simulation results for Setting 2 with $n=100$ and $p=200$.}
{
\begin{tabular}{cccccccccc}
\midrule
       & Oracle & STPS   & Univ   & CV     & DBMM   & ZZ     & Double & PODS   & DML   \\ \midrule
       & \multicolumn{9}{c}{$\beta_0$ and $\gamma_0$ are sparse, $R^2=0.5$}                         \\ \cmidrule{2-10} 
Bias   & -0.002 & -0.014 & -0.003 & -0.007 & -0.009 & -0.009 & -0.043 & -0.003 & 0     \\
SD     & 0.09   & 0.191  & 0.101  & 0.105  & 0.109  & 0.109  & 0.117  & 0.122  & 0.108 \\
Cover  & 0.931  & 0.962  & 0.89   & 0.913  & 0.936  & 0.933  & 0.907  & 0.908  & 0.938 \\
Length & 0.341  & 0.754  & 0.33   & 0.368  & 0.41   & 0.414  & 0.439  & 0.416  & 0.415 \\ \cmidrule{2-10} 
       & \multicolumn{9}{c}{$\beta_0$ and $\gamma_0$ are sparse, $R^2=0.9$}                         \\ \cmidrule{2-10} 
Bias   & -0.003 & -0.024 & -0.111 & -0.084 & -0.071 & -0.085 & -0.007 & 0.007  & 0.029 \\
SD     & 0.053  & 0.175  & 0.08   & 0.087  & 0.091  & 0.084  & 0.125  & 0.122  & 0.128 \\
Cover  & 0.935  & 0.964  & 0.534  & 0.734  & 0.823  & 0.75   & 0.937  & 0.903  & 0.98  \\
Length & 0.198  & 0.709  & 0.22   & 0.284  & 0.315  & 0.279  & 0.484  & 0.415  & 0.664 \\ \cmidrule{2-10} 
       & \multicolumn{9}{c}{$\beta_0$ is moderately sparse and $\gamma_0$ is dense, $R^2=0.3$}          \\ \cmidrule{2-10} 
Bias   & 0.003  & -0.013 & 0.019  & 0.006  & -0.002 & -0.003 & -0.061 & -0.014 & 0.008 \\
SD     & 0.101  & 0.192  & 0.102  & 0.109  & 0.115  & 0.112  & 0.125  & 0.134  & 0.11  \\
Cover  & 0.921  & 0.948  & 0.893  & 0.928  & 0.937  & 0.937  & 0.901  & 0.889  & 0.922 \\
Length & 0.359  & 0.744  & 0.342  & 0.39   & 0.434  & 0.427  & 0.466  & 0.437  & 0.398 \\ \cmidrule{2-10} 
       & \multicolumn{9}{c}{$\beta_0$ is moderately sparse and $\gamma_0$ is dense, $R^2=0.8$}          \\ \cmidrule{2-10} 
Bias   & 0.003  & -0.005 & 0.006  & 0.003  & 0.002  & 0.004  & -0.007 & -0.007 & 0.007 \\
SD     & 0.067  & 0.187  & 0.074  & 0.099  & 0.106  & 0.08   & 0.149  & 0.172  & 0.092 \\
Cover  & 0.921  & 0.962  & 0.867  & 0.931  & 0.94   & 0.898  & 0.923  & 0.827  & 0.937 \\
Length & 0.234  & 0.748  & 0.214  & 0.367  & 0.406  & 0.264  & 0.551  & 0.473  & 0.329 \\ \cmidrule{2-10} 
       & \multicolumn{9}{c}{$\beta_0$ and $\gamma_0$ are dense, $R^2=0.8$}                           \\ \cmidrule{2-10} 
Bias   & -      & 0.002  & 0.113  & 0.05   & 0.038  & 0.087  & -0.032 & -0.014 & 0.122 \\
SD     & -      & 0.187  & 0.097  & 0.106  & 0.111  & 0.096  & 0.149  & 0.173  & 0.092 \\
Cover  & -      & 0.967  & 0.41   & 0.877  & 0.917  & 0.66   & 0.911  & 0.819  & 0.657 \\
Length & -      & 0.773  & 0.221  & 0.379  & 0.42   & 0.273  & 0.547  & 0.472  & 0.338 \\ \midrule
\end{tabular}
}
\label{497}
\end{table}

\begin{table}
\centering
\caption{Simulation results for Setting 2 with $n=500$ and $p=1000$.}
{
\begin{tabular}{cccccccccc}
\midrule
       & Oracle & STPS   & Univ   & CV     & DBMM   & ZZ     & Double & PODS   & DML    \\ \midrule
       & \multicolumn{9}{c}{$\beta_0$ and $\gamma_0$ are sparse, $R^2=0.5$}                          \\ \cmidrule{2-10} 
Bias   & -0.001 & 0.001  & 0.002  & 0      & -0.001 & -0.001 & -0.026 & 0      & -0.002 \\
SD     & 0.04   & 0.085  & 0.045  & 0.045  & 0.047  & 0.049  & 0.051  & 0.052  & 0.047  \\
Cover  & 0.942  & 0.961  & 0.918  & 0.93   & 0.95   & 0.956  & 0.88   & 0.911  & 0.94   \\
Length & 0.152  & 0.34   & 0.156  & 0.164  & 0.183  & 0.196  & 0.187  & 0.179  & 0.177  \\ \cmidrule{2-10} 
       & \multicolumn{9}{c}{$\beta_0$ and $\gamma_0$ are sparse, $R^2=0.9$}                          \\ \cmidrule{2-10} 
Bias   & 0      & -0.002 & -0.046 & -0.035 & -0.025 & -0.025 & -0.024 & 0      & 0      \\
SD     & 0.023  & 0.083  & 0.037  & 0.039  & 0.042  & 0.041  & 0.051  & 0.053  & 0.047  \\
Cover  & 0.952  & 0.96   & 0.682  & 0.798  & 0.871  & 0.871  & 0.897  & 0.914  & 0.956  \\
Length & 0.089  & 0.341  & 0.123  & 0.138  & 0.154  & 0.154  & 0.191  & 0.179  & 0.189  \\ \cmidrule{2-10} 
       & \multicolumn{9}{c}{$\beta_0$ is moderately sparse and $\gamma_0$ is dense, $R^2=0.3$}        \\ \cmidrule{2-10} 
Bias   & -0.002 & 0.001  & 0.012  & 0.004  & 0.001  & 0.001  & -0.029 & -0.005 & -0.001 \\
SD     & 0.042  & 0.089  & 0.044  & 0.046  & 0.048  & 0.049  & 0.053  & 0.056  & 0.044  \\
Cover  & 0.933  & 0.959  & 0.91   & 0.936  & 0.952  & 0.953  & 0.891  & 0.91   & 0.942  \\
Length & 0.156  & 0.352  & 0.154  & 0.169  & 0.189  & 0.191  & 0.199  & 0.189  & 0.165  \\ \cmidrule{2-10} 
       & \multicolumn{9}{c}{$\beta_0$ is moderately sparse and $\gamma_0$ is dense, $R^2=0.8$}        \\ \cmidrule{2-10} 
Bias   & -0.001 & 0.002  & 0.006  & 0.002  & 0.002  & 0.004  & -0.011 & -0.003 & 0      \\
SD     & 0.025  & 0.093  & 0.029  & 0.04   & 0.043  & 0.032  & 0.061  & 0.067  & 0.037  \\
Cover  & 0.944  & 0.95   & 0.898  & 0.95   & 0.953  & 0.924  & 0.942  & 0.877  & 0.937  \\
Length & 0.095  & 0.37   & 0.097  & 0.156  & 0.173  & 0.116  & 0.231  & 0.209  & 0.134  \\ \cmidrule{2-10} 
       & \multicolumn{9}{c}{$\beta_0$ and $\gamma_0$ are dense, $R^2=0.8$}                            \\ \cmidrule{2-10} 
Bias   & -  & 0.004  & 0.16   & 0.068  & 0.049  & 0.123  & -0.035 & -0.012 & 0.144  \\
SD     & -  & 0.093  & 0.041  & 0.043  & 0.045  & 0.04   & 0.062  & 0.067  & 0.038  \\
Cover  & -      & 0.954  & 0.015  & 0.597  & 0.813  & 0.068  & 0.893  & 0.868  & 0.029  \\
Length & -   & 0.382  & 0.1    & 0.161  & 0.179  & 0.12   & 0.229  & 0.209  & 0.14   \\ \midrule
\end{tabular}

}
\label{498}
\end{table}

\section{Real Data Example}
In this section, we present an example of applying our method to real economic data.
We revisit the study of \cite{hr1978}, who investigated methodological problems in using housing data to estimate the demand for clean air. They used data for census tracts in the Boston Standard Metropolitan Statistical Area.
We examine the effect of nitrogen oxide concentration on housing prices by adding new control variables to the original regression model.

\cite{hr1978} considered the following linear regression model:
\begin{align*}
\log\text{MV}_i&= \alpha + \beta_1\text{NOX}_i^2 +\beta_2\text{RM}_i^2+\beta_3\log\text{DIS}_i+\beta_4\text{AGE}_i+\beta_5\log\text{RAD}_i+\beta_6\text{TAX}_i\\
&+\beta_7\text{PTRATIO}_i+\beta_8(B_i-0.63)^2+\beta_9\log\text{LSTAT}_i+\beta_{10}\text{CRIM}_i\\
&+\beta_{11}\text{ZN}_i+\beta_{12}\text{INDUS}_i +\beta_{13}\text{CHAS}_i +\epsilon_i,
\end{align*}
where MV denotes the median value of houses and NOX denotes the nitrogen oxide concentration in each census tract\footnote{In \cite{hr1978}, the exponent 2 of NOX is an estimated value rather than a predetermined value.}. 
See \cite{gp1996} for explanations of other covariates.
\cite{hr1978} showed that the OLS estimate of the coefficient of ${\rm NOX}^2$ has a negative sign, being highly significant.
Because the original data of \cite{hr1978} contained some incorrectly coded observations, \cite{gp1996} performed sensitivity analysis based on corrected data and showed that the conclusions of the empirical study do not change. \par

We performed sensitivity analysis by adding new control variables. 
We considered the following linear regression model:
\begin{equation}
\log\text{MV}_i =\alpha+ \beta_1 \text{NOX}_i^2 +X_{i,-1}^T \beta_{-1}+\epsilon_i, \label{443}
\end{equation}
where $X_{i,-1}$ denotes the vector of 78 control variables, including all covariates except for $\text{NOX}^2$ in the original model and their first order interaction terms. 
We used the data of \cite{gp1996}, which is available from the R package \verb|mlbench|.
The sample size was 506.
We performed statistical inference on $\beta_1$ using methods listed in our simulation studies. 
Because we could obtain the OLS estimate of $\beta_1$, we also compared the results of methods with OLS estimates in the original model and the model \eqref{443}. 
To obtain the estimator of DML, we employed the median method using 100 sample splits (see Definition 3.3 of \citealt{ccddhnr2018}). Because we do not know the true error variance in regressing $\text{NOX}^2_i$ on $X_{i,-1}$ for Univ, we estimate it by the sample mean of squared residuals from the node-wise Lasso tuned by the one standard error rule. \par

Table \ref{460} shows the results of applying methods to the model \eqref{443}.
The results of all methods are consistent with those of previous studies: all methods rejected the null hypothesis $\beta_1 =0$ at a 5\% significance level. 
Because the high-dimensional OLS estimate does not differ from the original OLS estimate, it is plausible to think that control variables in the original model are sufficient. 
The estimates of the debiased Lasso tuned by cross-validation, methods of Zhang and Zhang (2014), Dezeure et al. (2015), and STPS, and the estimate of PODS are close to the OLS estimates. 
Meanwhile, other estimates are relatively different from the two OLS estimates.  
The estimates of Univ and DML are relatively large, whereas the estimate of Double is small. 
Because the simulation studies show that STPS always has small biases when $n$ is large, estimates that significantly differ from that of STPS may have large biases in this real data example. 

\begin{table}[t]
\centering
\caption{Estimates of $\beta_1$ and their standard errors for the model \eqref{443} with $n=506$ and $\dim(X_{i,-1})=78$. OLS refers to the OLS results for $\beta_1$ in the original model, which are given in \cite{gp1996}. HOLS refers to the OLS results for $\beta_{1}$ in the model \eqref{443}.}
\begin{tabular}{cccccc}
\midrule
    & OLS      & HOLS     & STPS     & CV       & Univ     \\ \midrule
Est & -0.63724 & -0.65051 & -0.65184 & -0.63743 & -0.56759 \\
SE  & 0.11003  & 0.11512  & 0.14784  & 0.13480  & 0.08642  \\ \midrule
    & DBMM     & ZZ       & Double   & PODS     & DML      \\ \midrule
Est & -0.65184 & -0.65184 & -0.70252 & -0.66724 & -0.58113 \\
SE  & 0.14784  & 0.14784  & 0.14019  & 0.10230  & 0.15045  \\ \midrule
\end{tabular}
\label{460}
\end{table}

\section{Conclusion}
In this study, we analyzed theoretical and numerical properties of the debiased Lasso when the tuning parameter of the node-wise Lasso is much smaller than in previous studies.
Moreover, we proposed a data-driven tuning parameter selection procedure, called STPS.
 Our analysis shows that selecting a moderately small tuning parameter can mitigate the bias of the debiased Lasso without making variances diverge if the number of covariates is not too large.
 This implies that asymptotic normality for the debiased Lasso holds provided that the number of nonzero coefficients in linear regression models is $o(\sqrt{n/\log p})$. 
According to simulation studies and a real data example, STPS yields accurate confidence intervals. 
Therefore, STPS would be a useful tool for high-dimensional statistical inferences.  
 
\section*{Acknowledgments} 
We would like to thank valuable comments from Yoshimasa Uematsu, Mototsugu Shintani, Wenjie Wang, and participants at the autumn meeting of the Japanese Federation of Statistical Science Associations, the autumn meeting of the Japanese Economic Association,  the Miyagi meeting of the Kansai Econometric Society, and The 5th International Conference on Econometrics and Statistics in Kyoto. Akira Shinkyu is supported by JST SPRING Grant Number JPMJFS2126. Naoya Sueishi is supported by JSPS KAKENHI Grant Number 22H00833. 
\appendix
\section*{Appendix}
\begin{proof}[Proof of Lemma \ref{lem1}]
For the node-wise Lasso estimate $\hat{\gamma}_1=(\shgamma_{1,2},\ldots,\shgamma_{1,p})^T \in \mathbb{R}^{p-1}$, let $\hat{S}_1 = \{ j: \hat{\gamma}_{1,j} \neq 0\}$ and $\hat{s}_1 = |\hat{S}_1|$.
	Moreover, let $X_{\hat{S}_1}$ denote the submatrix of $X$ whose columns correspond to the elements of $\hat{S}_1$.
Because each column vector of $X$ is in general position, all column vectors of $X_{\lhs_1}$ are linearly independent, and the node-wise Lasso selects at most $n$ variables when $p>n$. 
Because $X_{\lhs_1}^TX_{\lhs_1}$ is invertible with probability one, for $\lambda_1 \ge C_1/ \sqrt{n}$, we can write
\begin{align*}
\ttau_1^2&=\fn X_1^TQ_{\lhs_1}X_1+\lambda_1^2\kappa_{\lhs_1}^T\left (\fn X_{\lhs_1}^TX_{\lhs_1}\right)^{-1}\kappa_{\lhs_1} \\
&\geq \lambda_{\min}\left(\fn X_{\{1\}\cup \lhs_1}^TX_{\{1\}\cup \lhs_1}\right)+\frac{\shs_1}{n}\frac{C_1^2}{\phi_{\max}(\shs_1)}\\
&\geq  \phi_{\min}(1+\shs_1)+\frac{\shs_1}{n}\frac{C_1^2}{\phi_{\max}(\shs_1)},
\end{align*}
where $Q_{\hat{S}_1}=I-X_{\hat{S}_1}(X_{\hat{S}_1}^TX_{\hat{S}_1})^{-1}X_{\hat{S}_1}^T$. The second inequality holds because of the definition of restricted maximum eigenvalues and the fact that $(X_1^TQ_{\lhs_1}X_1/n)^{-1}$ is the first diagonal element of $(X_{\{1\}\cup \lhs_1}^TX_{\{1\}\cup \lhs_1}/n)^{-1}$. See A-74 of \cite{g2012}. The last inequality holds by the definition of restricted minimum eigenvalues. Therefore, we obtain
\begin{equation*}
\frac{1}{\ttau_1^2}\leq \min \left\{\frac{1}{\phi_{\min}(1+\shs_1)}, \frac{n}{\shs_1}\frac{\phi_{\max}(\shs_1)}{C_1^2} \right \}.
\end{equation*}
Let $A_1=\{\phi_{\min}(n/K)\geq C_{\min}\}$, $A_2=\{\phi_{\max}(n)\leq C_{\max}\}$, $G$=\{each column vector of $X$ is in general position\}, and
\begin{equation*}
T=\left\{\frac{1}{\ttau_1^2}\leq \min \left\{\frac{1}{\phi_{\min}(1+\shs_1)}, \frac{n}{\shs_1}\frac{\phi_{\max}(\shs_1)}{C_1^2} \right \}\right\}.
\end{equation*}
Notice that $G\subset T\cap \{\shs_1 \leq n\}$. Let $B_1=\{\shs_1<n/K\}$ and $B_2=\{n \geq \shs_1\geq n/K\}$. Then, we have $\{\shs_1 \leq n\}=B_1\cup B_2$ and $B_1\cap B_2=\emptyset$. Thus, we have
\begin{equation}
A_1\cap A_2\cap G\subset A_1\cap A_2\cap T \cap (B_1\cup B_2)\subset (A_1 \cap B_1 \cap T)\cup (A_2 \cap B_2 \cap T). \label{652}
\end{equation}
First, we focus on the event $A_1\cap B_1 \cap T$. Recall that $n$ is divisible by $K$ so that $1+\shs_1 \leq n/K$. Recall also that $\phi_{\min}(s)$ is nonincreasing in $s$. Therefore, on the event $A_1\cap B_1 \cap T$, we have $\phi_{\min}(1+\shs_1)\geq \phi_{\min}(n/K)\geq C_{\min}$ and $1/\ttau_1^2 \leq 1/\phi_{\min}(1+\shs_1)$. Hence, we obtain 
\begin{equation}
A_1 \cap B_1 \cap T \subset \left\{\frac{1}{\ttau_1^2} \leq \frac{1}{C_{\min}} \right\}. \label{656}
\end{equation}
Next, we focus on the event $A_2 \cap B_2 \cap T$. Recall that $\phi_{\max}(s)$ is a nondecreasing function in $s$. Therefore, on the event $A_2 \cap B_2 \cap T$, we have $n\geq \shs_1\geq n/K$, $\phi_{\max}(\shs_1)\leq \phi_{\max}(n)\leq C_{\max}$, and $1/\ttau_1^2\leq n \phi_{\max}(\shs_1)/(\shs_1C_1^2)$. Hence, we obtain 
\begin{equation}
A_2 \cap B_2 \cap T \subset \left\{\frac{1}{\ttau_1^2}\leq \frac{KC_{\max}}{C_1^2} \right\}. \label{660}
\end{equation}
By combining \eqref{652}, \eqref{656}, and \eqref{660}, we have
\begin{equation*}
A_1\cap A_2\cap G\subset \left\{\frac{1}{\ttau_1^2} \leq \frac{1}{C_{\min}} \right\} \cup \left\{\frac{1}{\ttau_1^2}\leq \frac{KC_{\max}}{C_1^2} \right\}  \subset \left\{\frac{1}{\ttau_1^2}\leq \frac{1}{C_{\min}}+\frac{KC_{\max}}{C_1^2} \right\}.
\end{equation*}
By the union bound and Assumption 1, we obtain 
\begin{align*}
P\left(\frac{1}{\ttau_1^2}\leq \frac{1}{C_{\min}}+\frac{C_{\max}K}{C_1^2}\right)&\geq P(A_1 \cap A_2 \cap G)=1-P(A_1^c \cup A_2^c \cup G^c)\\
&\geq 1-P(A_1^c)-P(A_2^c).
\end{align*}
By Assumptions 2 and 3, we have
\begin{equation*}
\liminf_{n\rightarrow \infty}P\left(\frac{1}{\ttau_1^2}\leq \frac{1}{C_{\min}}+\frac{C_{\max}K}{C_1^2}\right )=1.
\end{equation*}
This completes the proof of this lemma.
\end{proof}

\begin{proof}[Proof of Proposition \ref{prop1}]
Assumptions 1 and 4 (ii) are satisfied from Lemma 4 of \cite{t2013} and Theorem 6 of \cite{rz2013}, respectively. 
We now show that the condition of Proposition 1 implies Assumptions 2, 3, and 4 (i). As in the appendix of \cite{jm2018}, who follow Remark 5.40 of \cite{v2012},  we have 
\begin{equation*}
P\left(\phi_{\max}(k)\geq K_{\max}+C\sqrt{\frac{k}{n}}+\frac{t}{\sqrt{n}}\right)\leq 2\begin{pmatrix}
p \\
k
\end{pmatrix}\exp(-ct^2),
\end{equation*}
for any $t \geq 0$ and $1 \le k \le n$, where positive constants $C$ and $c$ depend on $K_{\max}$ but not on $n$. 
Notice that 
\begin{equation*}
\begin{pmatrix}
p \\ k
\end{pmatrix}=\frac{p!}{k!(p-k)!}\leq \frac{p^k}{k!}
\leq \left(\frac{p}{k}\right)^k\exp(k)
\end{equation*}
because $\exp(k)=\sum_{i=0}^{\infty}k^i/i!\geq k^k/k!$ for any positive integer $k$. 
Thus, for $p \le C_0n$, we have 
\begin{align*}
P\left(\phi_{\max}(n)\geq K_{\max}+C+\frac{t}{\sqrt{n}}\right) \leq 2\exp(n(\log C_0+1)-ct^2).
\end{align*}
Let 
\begin{equation*}
t=\sqrt{\frac{2n(\log C_0+1)}{c}}.
\end{equation*}
Then,
\begin{equation*}
n(\log C_0+1)-ct^2=n(\log C_0+1)-2n(\log C_0+1)=-n(\log C_0+1).
\end{equation*}
Hence, we have 
\begin{equation*}
P\left(\phi_{\max}(n)\geq K_{\max}+C+\sqrt{\frac{2(\log C_0+1)}{c}}\right)\leq 2\exp(-n(\log C_0+1)).
\end{equation*}
By taking $C_{\max}=K_{\max}+C+\sqrt{2(\log C_0+1)/c}$, the first part of the proof is completed. The analogous arguments show that Assumption 4 (ii) is also satisfied. Similarly, following Remark 5.40 of \cite{v2012}, we have
\begin{equation*}
P\left(\phi_{\min}(k)\leq K_{\min}-C\sqrt{\frac{k}{n}}-\frac{t}{\sqrt{n}}\right)\leq 2 \begin{pmatrix}
p \\ k
\end{pmatrix} \exp(-ct^2).
\end{equation*}
Hence, we have 
\begin{align*}
P\left(\phi_{\min}\left(\frac{n}{K}\right)\leq K_{\min}-\frac{C}{\sqrt{K}}-\frac{t}{\sqrt{n}}\right) \leq 2\exp\left(\frac{n}{K}(\log K+\log C_0+1)-ct^2\right)
\end{align*}
for any $K$.
Let 
\begin{equation*}
t=\frac{\sqrt{n}K_{\min}}{2}.
\end{equation*}
Then, 
\begin{equation*}
P\left(\phi_{\min}\left(\frac{n}{K}\right)\leq \frac{K_{\min}}{2}-\frac{C}{\sqrt{K}}\right)\leq 2\exp\left(-n \left(\frac{cK_{\min}^2}{4}-\frac{1}{K}\{\log K+\log C_0+1\}\right)\right).
\end{equation*}
By taking $K$ so that
\begin{equation}
\frac{K_{\min}}{2}-\frac{C}{\sqrt{K}}>0, \label{663}
\end{equation}
we have
\begin{equation}
\frac{cK_{\min}^2}{4}-\frac{1}{K}\{\log K+\log C_0+1\}>0. \label{666}
\end{equation}
Notice that constants on left hand sides of \eqref{663} and \eqref{666} are independent of $n$. Thus, the second part of the proof is completed by taking $C_{\min}=K_{\min}/2-C/\sqrt{K}.$
\end{proof}

\begin{proof}[Proof of Theorem \ref{thm1}]
	For a sufficiently large positive constant $C$ such that $C\geq\sigma_\epsilon \sqrt{C^*}$, set $\lambda_0 =C\sqrt{2 \log p/n}$. Under Assumption 4, standard arguments of normal distribution yield 
	\begin{equation*}
	P\left(\frac{1}{n}\|X^T\epsilon\|_\infty > \lambda_0\right)=O\left(\frac{1}{\sqrt{\log p}}\right),
	\end{equation*}
	and we have the oracle inequalities
	\begin{equation*}
	\fn\|X(\hbeta-\beta_0)\|^2\leq \frac{8s_0\lambda_0^2 }{\kappa_{\min}^2},\ \ \|\hbeta-\beta_0\|_1\leq  \frac{4s_0\lambda_0 }{\kappa_{\min}^2},
	\end{equation*}
	with probability tending to one. See Lemmas 6.1 and 6.3 and Theorem 6.1 in \cite{bv2011}. Recall that 
	\begin{equation*}
	\Delta_1=\sqrt{n}(\hsigma\htheta_1-e_1)^T(\beta_0-\hbeta).
	\end{equation*}
	Thus, it follows from Lemma 1 that for $s_0$ such that $s_0=o(\sqrt{n/\log p})$
	\begin{equation*}
	|\Delta_1| \leq \sqrt{n} \frac{\lambda_1}{\htau_1^2}\|\hbeta-\beta_0\|_1= o_P(1).
	\end{equation*}
	Lemma 1 also yields
	\begin{equation*}
	\homega_{1,1}=\frac{\ttau_1^2}{\htau_1^4}\leq \frac{1}{\ttau_1^2}=O_P(1).
	\end{equation*}
	This completes the proof.
\end{proof}

\begin{proof}[Proof of Corollary \ref{cor1}]
The studentized debiased Lasso can be expressed as follows: 
\begin{equation*}
    \frac{\sqrt{n}(\sbhat_1-\beta_{01})}{\shsigma_\epsilon \homega_{1,1}^{1/2}}=\frac{1}{\shsigma_\epsilon \homega_{1,1}^{1/2}}\frac{1}{\sqrt{n}}\htheta_1{\! \!}^ TX^T\epsilon+ \tilde{\Delta}_1,
\end{equation*}
where
\begin{equation*}
    \tilde{\Delta}_1=\frac{\sqrt{n}(\hsigma\htheta_1-e_1)^T(\beta_0-\hbeta)}{\shsigma_\epsilon \homega_{1,1}^{1/2}}.
\end{equation*}
By the same argument in the proof of Theorem 1, we obtain $\tilde{\Delta}_1=o_P(1).$ Let 
\begin{equation*}
    Z_1=\frac{1}{\sigma_\epsilon \homega_{1,1}^{1/2}}\frac{1}{\sqrt{n}}\htheta_1{\! \!}^ TX^T\epsilon \sim N(0,1).
\end{equation*}
By the same argument of \cite{jm2014}, we obtain
\begin{align*}
    P\left(\frac{\sqrt{n}(\sbhat_1-\beta_{01})}{\shsigma_\epsilon \homega_{1,1}^{1/2}} \leq z\right)&=P\left(\frac{\sigma}{\shsigma_\epsilon}Z_1 + \tilde{\Delta}_1 \leq z\right)\leq P\left(Z_1\leq \frac{\shsigma_\epsilon }{\sigma_\epsilon}  (z+\delta)\right)+P\left(|\tilde{\Delta}_1|>\delta\right)\\
    &\leq \Phi(\delta (1+|z|)+\delta^2+z)+P\left(\left|\frac{\shsigma_\epsilon}{\sigma_\epsilon}-1\right|>\delta\right)+P(|\tilde{\Delta}_1|>\delta)
\end{align*}
for any $\delta>0$. By taking the limit supremum as $n\rightarrow \infty $ at first and the limit as $ \delta \rightarrow 0$ subsequently, we obtain  
\begin{equation*}
    \limsup_{n \rightarrow \infty} P\left(\frac{\sqrt{n}(\sbhat_1-\beta_{01})}{\shsigma_\epsilon \homega_{1,1}^{1/2}} \leq z\right) \leq \Phi(z).
\end{equation*}
Next, for any $\delta>0$ and $C>0$, we have
\begin{align*}
    \Phi(z-\delta)&=P(Z_1+ \delta \leq z )=P\left(\frac{\shsigma_\epsilon}{\sigma_\epsilon}\frac{1}{\shsigma_\epsilon \homega_{1,1}^{1/2}} \frac{1}{\sqrt{n}}\htheta_1{\!\!}^TX^T\epsilon +\delta \leq z \right)\\
    &\leq P\left(\frac{\sqrt{n}(\sbhat_1-\beta_{01})}{\shsigma_\epsilon \homega_{1,1}^{1/2}} \leq z\right) +P\left(|\tilde{\Delta}_1|>\frac{\delta}{2}\right)+P\left(\left|\frac{\shsigma_\epsilon}{\sigma_\epsilon}-1\right|>\frac{\delta}{2C^2}\right)\\
    &+P\left(\left|\frac{1}{\homega_{1,1}^{1/2}} \frac{1}{\sqrt{n}} \htheta_1{\!\!}^TX^T\epsilon \right|>C\right)+P\left(\frac{1}{\shsigma_\epsilon}>C\right).
\end{align*}
By taking the limit infimum as $n\rightarrow \infty$ at first, $C\rightarrow \infty$ subsequently, and $\delta \rightarrow 0$ afterward, we have 
\begin{equation*}
    \Phi(z)\leq \liminf_{n \rightarrow \infty} P\left(\frac{\sqrt{n}(\sbhat_1-\beta_{01})}{\shsigma_\epsilon \homega_{1,1}^{1/2}} \leq z\right).
\end{equation*}
This completes the proof.
\end{proof}

\begin{proof}[Proof of Theorem \ref{thm2}]
It is sufficient to show that $\tilde{\Delta}_1(\lambda_1)=o_P(1)$ and $\homega_{1,1}(\lambda_1)=O_P(1).$ Recall that $f(\lambda):=\|\hsigma\htheta_1(\lambda)-e_1\|_\infty/\homega_{1,1}^{1/2}(\lambda)$ and $\eta_1^*:=\ttau_1(\lambda_{\text{CV}})/\sqrt{n}$. 
First, suppose that $f(\lambda)> \eta_1^*$ for all $\lambda \in \Lambda$. Then, our procedure selects $\lambda_1=\lambda_{\min}$, and the KKT conditions for the node-wise Lasso yield
\begin{equation*}
	f(\lambda_1)= \frac{\|\hsigma\htheta_1(\lambda_{\min})-e_1\|_\infty}{\homega_{1,1}^{1/2}(\lambda_{\min})}\leq \frac{\lambda_{\min}}{\ttau_1(\lambda_{\min})}. 
\end{equation*}
Second, suppose that $f(\lambda)\leq \eta_1^*$ for some $\lambda \in \Lambda$. 
Because $\ttau_1^2(\lambda)$ is a nondecreasing function in $\lambda \in \Lambda$, we obtain
	\begin{equation*}
	f(\lambda_1)\leq f(\lambda^*)\leq \eta_1^*\leq \frac{\ttau(\lambda_{\max})}{\sqrt{n}}=\frac{1}{\sqrt{n}}\left(\frac{1}{n}\|X_1\|^2\right)^{1/2},
	\end{equation*}
	where the fourth equality follows from $\shgamma_1(\lambda_{\max})=0$ under Assumption 1.  \par

	Consequently, for the tuning parameter $\lambda_1$ selected by our procedure, we obtain 
	\begin{equation*}
	f(\lambda_1)\leq \max\left\{\frac{\lambda_{\min}}{\ttau_1(\lambda_{\min})}, \frac{1}{\sqrt{n}}\left(\frac{1}{n}\|X_1\|^2\right)^{1/2}\right\}=O_P\left(\frac{1}{\sqrt{n}}\right),
	\end{equation*}
	by Lemma 1. Therefore, by H\"{o}lder inequality,
	\begin{equation*}
	|\tilde{\Delta}_1(\lambda_1)|\leq f(\lambda_1)\frac{1}{\shsigma_\epsilon}\|\shbeta-\beta_0\|_1=O_P\left(\sqrt{\frac{s_0^2 \log p}{n}}\right)=o_P(1).
	\end{equation*}
	Moreover, we obtain
	\begin{equation*}
	\homega_{1,1}(\lambda_1)\leq \frac{1}{\ttau_1^2(\lambda_1)}\leq \frac{1}{\ttau_1^2(\lambda_{\min})}=O_P(1),
	\end{equation*}
	by Lemma 1. This completes the proof.
\end{proof}

\bibliographystyle{apalike} 
\bibliography{biblio}
\end{document}